\documentclass[11pt,epsfig]{article}
\usepackage{mathrsfs,bm, fullpage,amsmath,amsfonts,amsthm,graphicx,amssymb,textcomp,enumerate,multicol, thm-restate, bm, amsbsy}
 
\newtheorem{theorem}{Theorem}[section]
\newtheorem{corollary}[theorem]{Corollary}
\newtheorem{lemma}[theorem]{Lemma}
\newtheorem{proposition}[theorem]{Proposition}
\theoremstyle{remark}

\theoremstyle{definition}
\newtheorem{definition}[theorem]{Definition}
\title{Nielsen equivalence and trisections of 4-manifolds}
\author{Gabriel Islambouli}
\date{\vspace{-4ex}}

\begin{document}
\maketitle

\begin{abstract}
The goal of this paper is to construct distinct trisections of the same genus on a fixed 4-manifold. For every $k \geq 2$, we construct $2^{k}-1$ non-diffeomorphic $(3k,k)$-trisections on infinitely many 4-manifolds. Here, the manifolds are spun Seifert fiber spaces and the trisections come from Meier's spun trisections. The technique used to distinguish the trisections parallels an established technique for distinguishing Heegaard splittings. In particular, we show that the Nielsen classes of the generators of the fundamental group, obtained from spines of the 4-dimensional 1-handlebodies of the trisection, are isotopy invariants of the trisection. If we additionally consider the action of the automorphism group on the Nielsen classes, we obtain diffeomorphism invariants of trisections. 
\end{abstract}

\section{Introduction}

Heegaard splittings have long been used as a basic tool for understanding the topology of a 3-manifold. In attempting to understand a 3-manifold by its Heegaard splittings, the first task is the classification of its Heegaard splittings, up to a suitable equivalence. Perhaps the most famous in a large collection of such results is Waldhausen's theorem \cite{FW}, which states that $S^3$ has a unique Heegaard splitting in each genus, up to isotopy. In a similar vein, Bonahon and Otal showed that lens spaces have a unique splitting in each genus \cite{BO}. The first examples of non-isotopic Heegaard splittings of the same genus came in 1970, when Engmann constructed two non-isotopic Heegaard splittings of genus 2 on the connected sum of two lens spaces \cite{RE}.

Moving to 4 dimensions, a (g,k)-trisection is a decomposition of a smooth, closed 4-manifold into three pieces, each diffeomorphic to a 4-dimensional 1-handlebody. Trisections were introduced by Gay and Kirby in \cite{GK} as an extension of Heegaard splittings into 4-dimensions, and indeed, the theories seem to share structural similarities. Most notably, Gay and Kirby show, in the same work, that any two trisections of a 4-manifold become isotopic after some number of stabilizations, an operation akin to the stabilization of a Heegaard splitting. Despite this result, up to this point, there were not any examples of trisections where the stabilization operation was actually needed in order to make two different trisections isotopic. 


A construction of such examples is the main goal of this paper. We will show that techniques which work for Heegaard splittings of 3-manifolds also work for trisections of 4-manifolds. More precisely, we associate to each isotopy class of trisections three Nielsen classes of generators of the fundamental group, which we denote $\mathscr{N}(X_1), $ $\mathscr{N}(X_2), $ and $\mathscr{N}(X_3)$. This is akin to a Heegaard splitting, which is well known to admit two Nielsen classes; a fact which has been used by numerous authors in order to distinguish Heegaard splittings (see for example \cite{BCZ}, \cite{LM} and \cite{RE}).

To apply these invariants, we look to the spun trisections of spun 4-manifolds constructed by Meier in \cite{JM}. The construction takes, as input, a Heegaard splitting of a 3-manifold, $M^3 = H_1 \cup_\Sigma H_2$, and produces a trisection, $S(\Sigma)$, of  $S(M^3)$, the manifold obtained by spinning $M^3$. It is easy to show that $\pi_1(M^3) = \pi_1(S(M^3))$. Our main application hinges on a refinement of this result, which states that, in some sense, spinning a Heegaard splitting induces a Nielsen equivalence. As a result, we obtain the following theorem.

\begin{restatable*}{theorem}{mainTheorem}
\label{thm:mainTheorem}
Let $M^3$ be a closed, orientable, 3-dimensional manifold, and let $H_1 \cup_\Sigma H_2$ and $H_1' \cup_{\Sigma'} H_2'$ be Heegaard splittings of $M$. Then if $\mathscr{N}(H_1) \neq \mathscr{N}(H_1')$, the trisections $S(\Sigma)$ and $S(\Sigma')$ are not isotopic. Moreover, if for all $\phi \in Aut(\pi_1(M))$, $\phi(\mathscr{N}(H_1)) \neq \mathscr{N}(H_1')$ then the trisections $S(\Sigma)$ and $S(\Sigma')$ are not diffeomorphic.
\end{restatable*}

Leveraging the work done on Nielsen equivalence in Fuchsian groups by Lustig and Moriah in \cite{LM}, and Boileau Collins and Zieschang in \cite{BCZ}, we obtain the following corollary.

\begin{restatable*}{corollary}{nonIsoTri}
For every $k \geq 2$, there exist 4-manifolds which admit non-isotopic $(3k,k)$-trisections of minimal genus. 
\end{restatable*}

Having obtained non-isotopic trisections, we subsequently turn our attention to diffeomorphism classes of trisections. Here, our main input is a theorem of Lustig, Moriah, and Rosenberger in \cite{LMR}, which classifies generators of certain Fuchsian groups up to Nielsen equivalence and overall automorphisms of a group. Combining their results with Theorem \ref{thm:mainTheorem} gives us the following corollary.

\begin{restatable*}{corollary}{nonDiffTri}
For every $k \geq 2$, there exist 4-manifolds which admit non-diffeomorphic $(3k,k)$-trisections of minimal genus. 
\end{restatable*}

We conclude with a closer analysis of small Seifert fiber spaces. In this case, some curiosities arise. Namely, we exhibit 3 non-isotopic trisections on a fixed 4-manifold which become isotopic after a single stabilization. Nevertheless, these trisections remain non-isotopic under any series of unbalanced stabilizations (in the sense of \cite{MSZ}) in two of the sectors. 

\section{Nielsen equivalence}
We begin by reviewing the relevant group theory. We start by defining an equivalence relation between generating sets of the same size of a fixed group.
\label{sec:NE}
\begin{definition}
Let $F_n= F[x_1,...,x_n]$ be the free group on $n$ elements with basis $(x_1,...,x_n)$. Let $G$ be a finitely generated group and $A=(a_1,...,a_n)$ and $B=(b_1,....b_n)$ be generating sets of size $n$ for  $G$. $A$ and $B$ are called \textbf{Nielsen equivalent} if there exists a basis $(y_1,y_2,..., y_n)$ for $F[x_1,...x_n]$ and a homomorphism $\phi:F_n \to G$ so that $\phi(x_i) = a_i$ and $\phi(y_i) = b_i$. 
\end{definition}

It is a classical result \cite{JN} that the automorphism group of the free group is generated by the elementary Nielsen transformations. Given the free group on $n$ elements with ordered basis $(x_1,x_2,....,x_n)$, the Nielsen transformations are the following:

\begin{enumerate}
\item Swap $x_1$ and $x_2$.
\item Cyclically permute $(x_1,x_2,....,x_n)$ to $(x_2,x_3,....,x_n, x_1)$.
\item Replace $x_1$ with $x_1^{-1}$.
\item Replace $x_1$ with $x_1x_2$.
\end{enumerate}

In light of this result, we obtain an alternative characterization of Nielsen equivalence. Given two generating sets of a group, $A=(a_1,...,a_n)$ and $B=(b_1,....b_n)$, write the $b_i$ as words, $w_i$, in the generators $(a_1,a_2,...,a_n)$ to obtain an ordered list of words, $(w_1, w_2,..., w_n)$. $A$ and $B$ are Nielsen equivalent if and only if we can successively apply the automorphisms 1-4 above to get from the ordered set $(a_1,a_2,...,a_n)$ to $(w_1, w_2,..., w_n)$. Here, we are allowed to simplify words using any applicable relations in the group.

\section{Spines of handlebodies}

For $n \geq 3$, we define an \textbf{n-dimensional handlebody of genus g} to be the unique smooth, orientable manifold obtained by attaching $g$ 1-handles to an n-dimensional ball. We will denote this manifold by $\bm{H_g^n}$. Handlebodies deformation retract onto graphs embedded within them. We call any embedded graph which is a deformation retract of $H_g^n$ a \textbf{spine} of $H_g^n$. Most naturally, one can construct a spine by connecting the cores of the 1-handles to a common point in the interior of the n-ball.

In general, a spine may contain an arbitrary number of vertices, however, in this paper, we will be particularly concerned with spines which have a single vertex. In addition, our spines will come with an orientation of each edge. Henceforth, all spines of handlebodies will be assumed to have one vertex and carry an orientation, unless otherwise noted. In this case, a spine of $H_g^n$ will be a wedge of g circles with each circle carrying an orientation. Spines are considered up to base point preserving isotopy in $H_g^n$.  If $g<2$, then a handlebody has a unique spine, but otherwise, there are infinitely many spines.

A spine of a handlebody can be altered in controlled ways to obtain a new spine of the handlebody. One may reverse the orientation of any edge, and the resulting graph is clearly still a spine. A more interesting move on one vertex spines is what is called an \textbf{edge slide}. Informally, this amounts to sliding the end of a loop over another loop in the direction of the second loops orientation and returning back to the base vertex. This is illustrated in Figure \ref{fig:edgeSlide}. 

Though not very enlightening, we also give a formal definition of an edge slide. Let $S$ be a spine of $H^n_g$, and let $l_1$ and $l_2$ be two loops of $S$ parameterized by $f_1, f_2:[0,1] \to H^n_g$. Let $g:D^2 \to H^n_g$ be an embedding of a 2-dimensional disk so that the boundary $S^1$ is parameterized by $[0,1]$ and oriented in the direction of increasing real number values with $h = g|_{S^1}$. Suppose that $h([0,\frac{1}{3}]) = f_1([\frac{2}{3},1])$ and $h([\frac{1}{3},\frac{2}{3}]) = f_2([0,1])$ where both restricted functions are orientation preserving homeomorphisms of the interval. We may obtain a new one vertex spine of $H^n_g$ by leaving all edges of G unchanged except for $l_1$ which is replaced by $(l_1 \backslash h([0,\frac{1}{3}])) \cup \overline{h([\frac{2}{3},1])}$ (where the bar indicates opposite orientation). This process is called sliding $l_1$ over $l_2$.

\begin{figure}
\label{fig:edgeSlide}
\includegraphics[scale=.55]{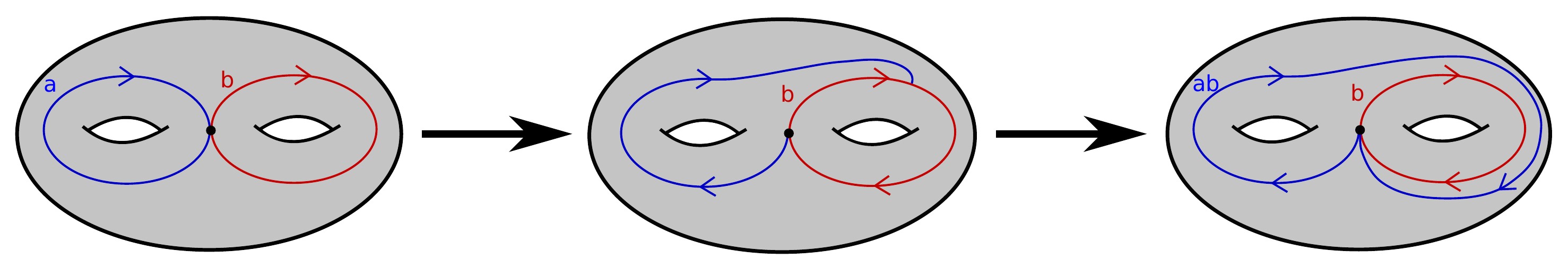}
\caption{The process of sliding an edge which represents the generator $a$ over an edge representing the generator $b$ produces a spine which has edges representing the generators $ab$ and $b.$}
\end{figure}

It is straightforward to see that $\pi_1(H^n_g)$ is a free group on $g$ generators. Moreover, any spine of $H^n_g$ specifies a basis for $F_n$. The following lemma relates the moves on spines discussed above to the Nielsen transformations discussed in Section \ref{sec:NE} and is straightforward to prove.

\begin{lemma}
\label{lem:NEareEM}
Let $S$ be a spine of $H^n_g$ consisting of loops labeled $x_1,...,x_g$. Let $\pi_1(H^n_g) = F_g$ have the ordered basis consisting of the labels $(x_1,...,x_n)$. Then Nielsen transformations on $F_g$ of type 1 and 2 correspond to relabeling the edges of $S$. Type 3 transformations correspond to reversing the orientation of an edge and type 4 transformations correspond to edge slides. Therefore, all automorphisms of $\pi_1(H^n_g)$ are realized by permutations of labels, edge slides, and reversals of orientations on spines.
\end{lemma}

The previous lemma shows that any ordered basis of the free group can be realized as a spine of a given handlebody. The next logical question is the uniqueness of this realization. The following lemma shows that each basis is realized by a unique spine.

\begin{lemma}
\label{lem:AllSpinesAreESEQ}
Any two spines of $H^n_g$ are related by changing orientations and edge slides.
\end{lemma}

\begin{proof}
In dimension 3, this is well known, but we sketch a proof for completeness. In a 3-dimensional handlebody, a spine gives rise to a unique set of disks dual to the spine. In our situation, where the spine has a single vertex, we get a minimal disk system, that is, a collection of disks which cuts the handlebody into a single 3-ball. Conversely, a minimal disk system also gives rise to a unique spine. It follows from the work of Reidemeister and Singer in \cite{KR} and \cite{JS} that any two minimal disk systems for a handlebody are related by a sequence of disk slides. If $D_1$ and $D_2$ are dual disks for $l_1$ and $l_2$, respectively, then the disk system obtained by disk sliding $D_1$ over $D_2$ is dual to the spine obtained by edge sliding $l_2$ over $l_1$. Thus any spine can be obtained by converting to disk systems and performing the dual edge slides prescribed by the disk slides between the disk systems.

In dimensions $n \geq 4$, the situation is simpler. Let $S$ and $S'$ be spines consisting of loops $s_1,...,s_g$ and $s'_1,...,s'_g$ respectively. The homotopy classes of these loops specify 2 bases for $F_g$. These bases are related by Nielsen transformations which, by Lemma \ref{lem:NEareEM}, can be realized geometrically as moves on spines. Apply these moves to $S$ until we obtain loops in the same homotopy class as the loops in $S'$. Now since $n \geq 4$ homotopic loops are in fact isotopic. Then loop of $S$ can be isotoped to the corresponding loop of $L'$ homotopic to it by an isotopy which, by general position, misses the other loops.
\end{proof}

\section{Decompositions of manifolds}

In this section, we will briefly review decompositions of 3- and 4-manifolds into handlebodies as well as equivalence relations between these decompositions. It is worth noting that these decompositions have been generalized to piecewise linear manifolds of arbitrary dimension in \cite{RT}. We begin with dimension 3.

\subsection{Heegaard splittings of 3-manifolds}

\begin{definition}
A \textbf{genus g Heegaard splitting} of a 3-manifold $M^3$ is a genus g surface, $\Sigma_g$, embedded in $M^3$ so that the complement of an open regular neighborhood of $\Sigma_g$ in $M^3$ has two components, $H_1$ and $H_2$, each homeomorphic to $H_g^3$. Two Heegaard splittings defined by $H_1 \cup_{\Sigma} H_2$ and $H_1' \cup_{\Sigma'} H_2'$ are \textbf{isotopic as Heegaard splittings} if $\Sigma$ is isotopic $\Sigma'$ by an isotopy taking $H_1$ to $H_1'$ and $H_2$ to $H_2'$. Two Heegaard splittings, denoted as above, are \textbf{homeomorphic as Heegaard splittings} if there is a homeomorphism $f:M^3 \to M^3$ so that $f(H_1) = H_1'$ and $f(H_2) = H_2'$.
\end{definition}

A Heegaard splitting can naturally be seen as the middle level of a self indexing Morse function. From this point of view, $H_1$ is the union of the 0- and 1-handles, so that the inclusion map $i: H_1 \hookrightarrow M^3$ induces a surjection of the fundamental groups. The fundamental group of $H_1$ is likewise surjected onto by the fundamental group of its spine under the inclusion map. By inverting the Morse function, $H_2$ becomes the union of the 0- and 1-handles, so the fundamental group of the spine of $H_2$ also surjects onto $\pi_1(M^3)$. Thus, a Heegaard splitting determines two sets of generators for $\pi_1(M^3)$. In light of Lemmas \ref{lem:NEareEM} and \ref{lem:AllSpinesAreESEQ}, the spines of each handlebody give well defined Nielsen equivalence classes of $\pi_1(M^3)$. This motivates the following definition, which will also be used in decompositions in other dimensions.

\begin{definition} 

Let $i: H_g^n \hookrightarrow M^n$ be an embedding of an n-dimensional genus g handlebody which induces a surjection of fundamental groups and let $H$ denote the image of $ H_g^n$ in $M^n$. We will denote by $\pmb{\mathscr{N}}\bm{(H)}$ the Nielsen class of the generators of $\pi_1(M^n)$ obtained from spines of $H$. If $f$ is a self homeomorphism of $M^n$ we will denote by $\bm{f}(\pmb{\mathscr{N}}\bm{(H))}$ the Nielsen class obtained by applying the induced map on fundamental groups to each loop of some spine of H. 

\end{definition}

The following proposition has been used extensively to distinguish Heegaard splittings (see \cite{BCZ} and \cite{LM} for particular applications).

\begin{proposition}
\label{prop:HSIsoImpliesNE}
Let $H_1 \cup_{\Sigma} H_2$ and $H_1' \cup_{\Sigma'} H_2'$ be two genus g Heegaard splittings of $M^3$. If these Heegaard splittings are isotopic, then $\mathscr{N}(H_1) = \mathscr{N}(H_1')$, and $\mathscr{N}(H_2) = \mathscr{N}(H_2')$. If the Heegaard splittings are homeomorphic by some homeomorphism $f$, then $f(\mathscr{N}(H_1)) = \mathscr{N}(H_1')$ and $f(\mathscr{N}(H_2)) = \mathscr{N}(H_2')$.
\end{proposition}

\begin{proof}
If $H_1 \cup_{\Sigma} H_2$ and $H_1' \cup_{\Sigma'} H_2'$ are isotopic, then the isotopy takes $H_1$ to $H_1'$. In particular, a spine for $H_1$ is taken to a spine for $H_1'$. By Lemma \ref{lem:AllSpinesAreESEQ} these spines can be made equivalent by a series of edge slides and reversals of orientations. By Lemma \ref{lem:NEareEM} this implies that the generators of $\pi_1(M^3)$ coming from the spines are Nielsen equivalent. An identical argument shows that $\mathscr{N}(H_2) = \mathscr{N}(H_2')$. If the splittings are homeomorphic, a similar argument applies after the application of the homeomorphism. 
\end{proof}

\subsection{Trisections of 4-manifolds}

We now turn our attention to dimension 4. We again seek to decompose an arbitrary smooth, closed, orientable manifold into handlebodies, but simply repeating the construction in dimension 3 yields a triviality. The only 4-dimensional manifold which is the union of 2 copies of $H_g^4$ is $\#^g S^1 \times S^3$, as can be quickly derived from a theorem of Laudenbach and Poenaru \cite{LP}. If instead, however, we allow ourselves  3 copies of $H_k^4$, glued appropriately, the theory becomes much richer. We begin with a more precise definition.
\begin{definition}
A $\bm{(g,k)}$\textbf{-trisection} of a 4-manifold, $M$, is a decomposition $M = X_1 \cup X_2 \cup X_3$ such that: 
\begin{itemize}
\item $X_i \cong  H_k^4$
\item $X_i \cap X_j = H_{ij } \cong H_g^3$ for $i \neq j$
\item $\partial X_i = H_{ij}\cup H_{ik}$ is a genus g Heegaard splitting for  $\partial X_i = \#^kS^1 \times S^2$
\end{itemize}
Two trisections of a fixed 4-manifold, $M$, defined by $X_1 \cup X_2 \cup X_3$ and $X_1' \cup X_2' \cup X_3'$ are \textbf{diffeomorphic as trisections} if there is a diffeomorphism of $M$ such that $f(X_i)= X_i'$. Two trisections are \textbf{isotopic as trisections} if there is an isotopy, $f_t$, of $M$ such that $f_0 = id$ and $f_1(X_i)= X_i'$.
\end{definition}

In \cite{GK}, Gay and Kirby show that every closed, smooth, orientable 4-manifold admits a trisection. Additionally, they show that any two trisections of a fixed 4-manifold become isotopic after some number of a stabilization operation. The reader is referred to their paper for the details of this operation, as well as an introduction into the theory. 

In Lemma 13 of \cite{GK}, the authors also show how to obtain a handle decomposition of a 4-manifold, M, from a trisection, $M=X_1 \cup X_2 \cup X_3$, so that $X_1$ is the union of the 0- and 1-handles. From this point of view, it is clear that $X_1$ generates $\pi_1(M)$. Moreover, it follows from the definition of a trisection that one may permute the labels of the $X_i$ arbitrarily, and the result is still a trisection. As a result, we may in fact consider any of the $X_i$ to be the union of the 0- and 1-handles in some handle decomposition of $M$, so that each $X_i$ gives rise to a set of generators of $\pi_1(M)$. Using Lemmas \ref{lem:AllSpinesAreESEQ} and \ref{lem:NEareEM}, we see that a trisection gives rise to three Nielsen equivalence classes of generators. The following proposition and its proof parallel Proposition \ref{prop:HSIsoImpliesNE}. 

\begin{proposition}
\label{prop:triIsoImpliesNE}
Let $X_1 \cup X_2 \cup X_3$ and $X_1' \cup X_2' \cup X_3'$ be two $(g,k)$-trisections of $M^4$. If these trisections are isotopic, then $\mathscr{N}(X_i) = \mathscr{N}(X_i')$. If the trisections are diffeomorphic by some diffeomorphism $f$, then $f(\mathscr{N}(X_i)) = \mathscr{N}(X_i')$.
\end{proposition}

\begin{proof}
If $X_1 \cup X_2 \cup X_3$ and $X_1' \cup X_2' \cup X_3'$ are isotopic, then in particular a spine for $X_i$ is taken to a spine for $X_i'$. By Lemma \ref{lem:AllSpinesAreESEQ} these spines are related by a series of edge slides and reversals of orientations. By Lemma \ref{lem:NEareEM} this implies that the generators of $\pi_1(M^4)$ coming from the spines are Nielsen equivalent. If the splittings are homeomorphic, a similar argument applies after the application of the homeomorphism. 
\end{proof}

\section{Meier's spun trisections}
\label{sec:spinConstruction}
Let $M$ be a closed 3-manifold, and denote by $M_\circ$ the punctured manifold, obtained by removing a 3-ball from $M$. The boundary of $M_\circ \times S^1$ is $S^2 \times S^1$, which we may fill in with $S^2 \times D^2$ in two ways. Let $S(M) =  M_\circ \cup_{id} S^2 \times D^2$ be the result of capping off  $M_\circ$ with $S^2 \times D^2$ via the identity map, and let $S^*(M) = M_\circ \cup_{\tau} S^2 \times D^2$ be the result of capping off  $M_\circ$ with $S^2 \times D^2$ via the unique self homeomorphism of $S^2 \times S^1$ which does not extend across $S^2 \times D^2$. In other words, $S(M)$ and $S^*(M)$ differ by a Gluck twist about the copy of $S^2$ being attached. It is straightforward to see that $\pi_1(S(M)) = \pi_1(S^*(M)) = \pi_1(M)$.

In \cite{JM}, Meier gives a construction which, given a genus $g$ Heegaard splitting, $M = H_1 \cup_\Sigma H_2$, produces $(3g,g)$-trisections, $S(\Sigma)$ and $S^*(\Sigma)$, of $S(M)$ and $S^*(M)$ respectively. Shortly thereafter, Hayano \cite{KH} showed that $S(\Sigma)$ and $S^*(\Sigma)$ are simplified trisections, as defined by Baykur and Saeki in \cite{BS}. We will briefly sketch the construction, and refer the reader to the proof of Theorem 1.2 of \cite{JM} for a more in depth treatment. 

Begin with a genus $g$ Heegaard splitting, $M=H_1 \cup_\Sigma H_2$. This splitting can be used to define a Morse function, $f:M \to [0,2]$, with a unique index 0 critical point with critical value $0$, $g$ index 1 critical points which take distinct values in $(0,1)$, $g$ index 2 critical points which take distinct values in $(1,2)$, and a unique index 3 critical point with critical value $2$. Such a function can be taken such that $H_1 = f^{-1}([0,1])$ and $H_2 = f^{-1}([1,2])$.

Next, remove the 3-handle corresponding to the index 3 critical point in order to puncture $M$ in $H_2$, and take the product with $S^1$ to obtain $M_\circ \times S^1$. If we parameterize a disk of radius $2$ using polar coordinates, we obtain a generic smooth function $\tilde{f}: M_\circ \times S^1 \to D^2$ defined by $\tilde{f}(x,\theta) = (2-f(x), \theta)$. Here, we take $2-f(x)$ in the first coordinate so that the ``missing" piece of $M$ is located in the center of the disk. One can then fill in the boundary component of $M_\circ \times S^1$ with $S^2 \times D^2$ and extend $\tilde{f}$ across the inner disk of $D^2$ in the obvious way. At this point, each index $i$ critical point of $f$ gives rise to an index $i$ critical fold of $\tilde{f}$ as we move from the outer edge of the disk towards the center. We can use a sequence of always realizable moves on the index $2$ critical folds coming from $H_2$ in order to ``flip'' them so that all folds are of index 1 when moving towards the center of the disk. This turns each index $2$ critical fold into an index $1$ critical fold with $6$ cusps.
 
After these moves, we can trisect the disk in the obvious way to obtain a $(3g,g)$-trisection of $S(M)$. An important observation made in \cite{JM} is that, throughout this process, $\tilde{f} (H_1 \times S^1)$ is left unaltered. The gluing takes place away from this region, and the modifications to the Morse 2-function all happen on the folds coming from the critical points of $H_2$. In addition, $H_1$ was left unaltered when passing from $M$ to $M_\circ$. Therefore, the Morse 2-function for $S(\Sigma)$ can be decomposed into two pieces; one of which is  $H_1 \times S^1$, and the other of which is the spin of $H_2$. These two pieces meet along an embedded copy of $\Sigma \times S^1$. This decomposition is shown in Figure \ref{fig:spinMorse2}.

\begin{figure}

\begin{center}
\includegraphics[scale=.55]{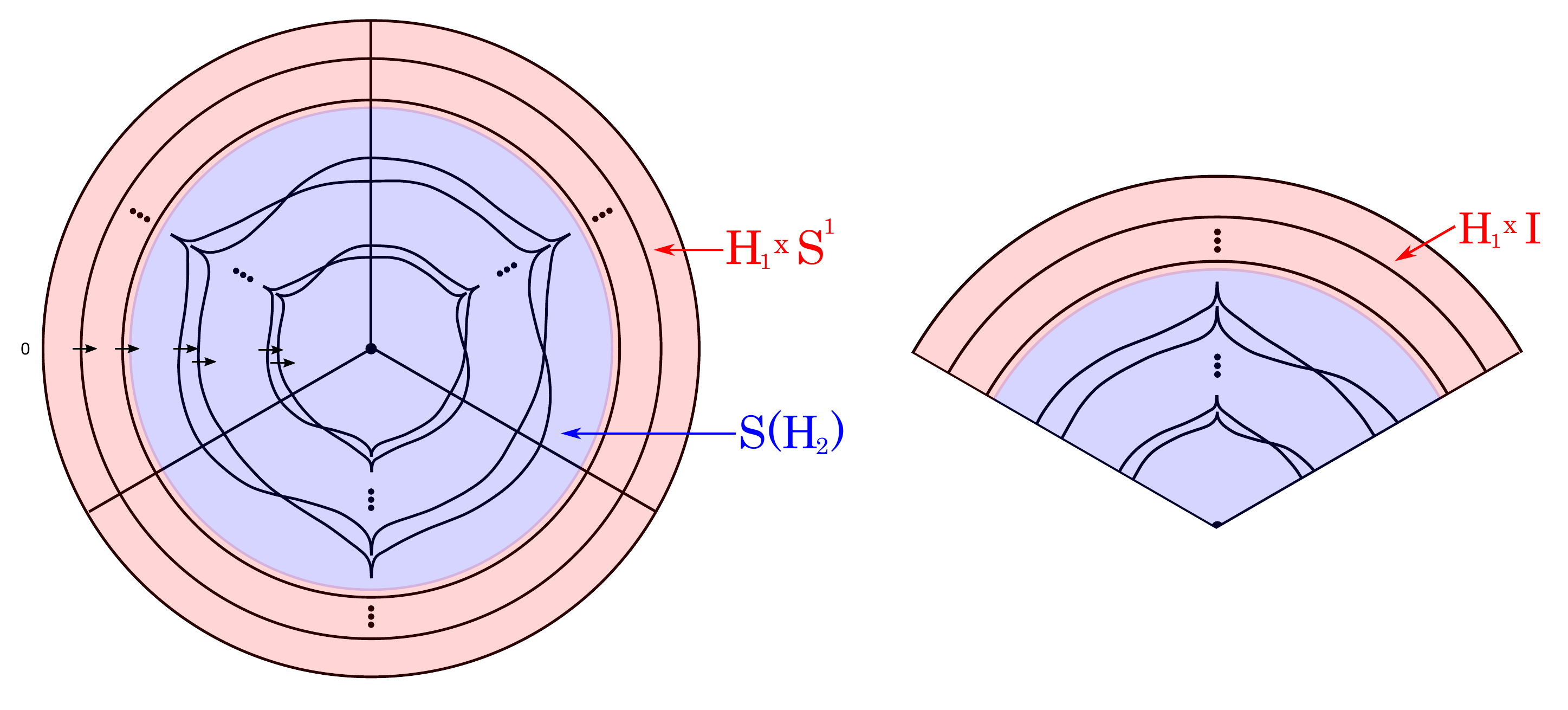}
\end{center}
\caption{Left: The trisected Morse 2 function for $S(\Sigma)$. Right: The portion without cusps of $X_i$ can be identified with $H_1 \times I$}
\label{fig:spinMorse2}
\end{figure}

\section{Proof of the main theorem}
Throughout this section, $M = H_1 \cup_\Sigma H_2$ will be a Heegaard splitting, and $S(\Sigma) = X_1 \cup X_2 \cup X_3$ will be the trisection of $S(M)$ constructed in Section \ref{sec:spinConstruction}. By the decomposition of the Morse 2-function noted in the previous section, we have an $S^1$ parameterized family of embeddings $f_\theta:H_1 \to S(M)$. Let $f_{[a]} = f_{a}(H_1) \subset S(M)$ and $f_{[a, b]}$ be $\cup_{i=a}^b f_{[i]}$. We parameterize $\theta$ such that $\theta \in [0,1)$, $f_{[0, \frac{1}{3}]} \subset X_1$, $f_{[\frac{1}{3}, \frac{2}{3}]} \subset X_2$, $f_{[\frac{2}{3},1) \cup \{0\}]} \subset X_3$.

\begin{lemma}
\label{lem:spineOfHSisSpineOfTri}
Let $H_1 \cup_\Sigma H_2$ be a genus $g$ Heegaard splitting of a closed 3-manifold $M$ and let $S(\Sigma) = X_1 \cup X_2 \cup X_3$ be the $(3g,g)$-trisection constructed in Section \ref{sec:spinConstruction}.  Then if $S$ is a spine for $H_1$, $f_{\frac{1}{6}}(S)$ is a spine for $X_1$, $f_{\frac{1}{2}}(S)$ is a spine for $X_1$, and $f_{\frac{5}{6}}(S)$ is a spine for $X_3$, 
\end{lemma} 

\begin{proof}
We will show that $f_{\frac{1}{6}}(S)$ is a spine for $X_1$, as the other two claims have identical proofs. $X_1$ consists of two pieces: one of the pieces is one third of $H_1 \times S^1$, which we will call $\tilde{X}_1$, and the other is one third of $S(H_2)$, which we will call $\tilde{X}_1^c$ (see Figure \ref{fig:spinMorse2}). The critical folds of $\tilde{X}_1^c$ are all cusped and so $\tilde{X}_1^c$ is a collar on $\tilde{X}_1$. Therefore, there is a deformation retraction of $X_1$, as a subset of $S(M)$, onto $\tilde{X_1}$. $\tilde{X_1}$, in turn, can be viewed as $f_{[0, \frac{1}{3}]}$ which can be identified with $f_{\frac{1}{6}} \times [-\frac{1}{6},\frac{1}{6}]$ using $f_\theta$. We therefore conclude that $\tilde{X_1}$ retracts onto $f_{\frac{1}{6}}(H_1)$. This is an embedded copy of $H_1$, and so we may compose the retraction of $H_1$ onto $S$ with $f_{\frac{1}{6}}$ to obtain a retraction from $f_{\frac{1}{6}}(H_1)$ onto $f_{\frac{1}{6}}(S)$. After composing all of these retractions, we obtain that $X_1$ retracts onto $f_{\frac{1}{6}}(S)$.
\end{proof}

The previous lemma shows us how to obtain spines for the $X_i$ from a spine for $H_1$. Since we seek to eventually distinguish trisections by the Nielsen classes of their spines, we must determine how these spines generate the fundamental group of $S(M)$. The following lemma shows that the generators coming from the spines for the $X_i$ are essentially the same as the generators coming from a spine for $H_1$.

\begin{lemma}
\label{lem:spinIsANE}
Let $S$ be a spine for $H_1$, and let $P$ be the presentation for $\pi_1(M)$ given by $\langle s_1,  ... ,s_g |R \rangle$, where the $s_i$ are the homotopy classes of the loops which make up $S$, and $R$ is a set of relations induced by $H_2$. Then for any $\theta_\circ \in [0,1)$, the set of generators given by the loops of $f_{\theta_\circ}(S)$ induce the same presentation for $\pi_1(S(M))$.
\end{lemma}

\begin{proof}
We will break down each step of the construction of $S(M)$ and track presentations of the fundamental groups along the way. The first step is to puncture $M$ in $H_2$ to obtain $M_\circ$. There is a natural inclusion $i: M_\circ \hookrightarrow M$. This induces an isomorphism, $i_*$, on fundamental groups, and so there is an inverse $i_*^{-1}: \pi_1(M) \to \pi_1(M_\circ)$. Since $S$ is left unchanged, and generates $\pi_1(M_\circ)$, we may associate $\pi_1(M_\circ)$ with $P$ so that $i_*^{-1}(s_i) = s_i$. Also, since $H_1$ is unaltered in this process, the restriction $i|_{H_1}$ is a homeomorphism onto its image, so it has an inverse, $i|_{H_1}^{-1}$.

Next, parameterize $S^1$ by  $\theta \in [0,1)$. Fix an angle $\theta_\circ \in [0,1)$ and denote by $j: M_\circ \hookrightarrow M_\circ \times S^1$ the inclusion given by $j(x) = (x,\theta_\circ)$. We may associate $\pi_1(M_\circ \times S^1)$ with the presentation of $\pi_1(M) \times \mathbb{Z}$ given by  $\langle s_1, s_2, ... ,s_g, z |R, zs_i = s_iz \rangle$ so that $j_*(s_i) = s_i$. Finally, we have an inclusion map $k: M_\circ \times S^1 \hookrightarrow S(M)$. Here, $k_*$ induces the projection map $\pi_1(M) \times \mathbb{Z} \to \pi_1(M)$.  We may associate $\pi_1(S(M))$ with $P$ so that $k_*(s_i, z) = s_i$.

Now note that $k \circ j \circ i|_{H_1}^{-1} = f_\theta$. Moreover we have arranged matters so that if we identify both $\pi_1(M)$ and $\pi_1(S(M))$ with the presentation $P$, then $f_*(s_i) = s_i$, as desired.

\end{proof}

As far as the group theory is concerned, we now have a complete understanding of the spines of the $X_i$ in relation to a spine for $H_1$. This understanding, together with Proposition \ref{prop:triIsoImpliesNE} are the main inputs used to prove the main theorem.

\mainTheorem

\begin{proof}
By an isotopy of $\Sigma$, we may arrange so that $H_2 \cap H_2' \neq \emptyset$. We may then use the spin construction to obtain two trisections of $S(M)$, given by $S(\Sigma) = X_1 \cup X_2 \cup X_3$ and $S(\Sigma')= X_1' \cup X_2' \cup X_3'$, making sure to puncture $M$ in $H_2 \cap H_2'$ such that the removed ball intersects neither $H_1$ nor $H_1'$. Let $S$ and $S'$ be spines of $H_1$ and $H_1'$ respectively, and let $f_\theta$ and $g_\theta$ be the $S^1$ parameterized families of embeddings of $H_1$ and $H_1'$, respectively, into $S(M)$. 

By Lemma \ref{lem:spineOfHSisSpineOfTri}, $f_{\frac{1}{6}}(S)$ is a spine for $X_1$, and $g_{\frac{1}{6}}(S')$ is a spine for $X_1'$. By Lemma \ref{lem:spinIsANE}, we may identify the fundamental groups of $M$ and $S(M)$ such that these spines for $X_1$ and $X_1'$ induce the same sets of generators as the spines for $H_1$ and $H_1'$, respectively. Therefore, $\mathscr{N}(H_1) = \mathscr{N}(X_1)$ and $\mathscr{N}(H_1') = \mathscr{N}(X_1')$. By assumption,  $\mathscr{N}(H_1) \neq \mathscr{N}(H_1')$ so that $\mathscr{N}(X_1) \neq \mathscr{N}(X_1')$. By Proposition \ref{prop:triIsoImpliesNE}, $S(\Sigma)$ and $S(\Sigma')$ are not isotopic. 

Similarly, let us suppose towards a contradiction that $S(\Sigma)$ and $S(\Sigma')$ are diffeomorphic by some diffeomorphism $h$. Then, by the second part of Proposition \ref{prop:triIsoImpliesNE}, $h(\mathscr{N}(X_1)) = \mathscr{N}(X_1')$. But since $\mathscr{N}(H_1) = \mathscr{N}(X_1)$ and $\mathscr{N}(H_1') = \mathscr{N}(X_1')$ this implies that $h(\mathscr{N}(H_1)) = \mathscr{N}(H_1')$. This contradicts the assumption that no such $h$ exists.

\end{proof}

\section{Heegaard splittings of Seifert fiber spaces}
\subsection{Vertical Heegaard splittings and non-isotopic trisections}
In this section, we will restrict ourselves to orientable Seifert fiber spaces with orientable base spaces, which are sometimes referred to as fully orientable Seifert fiber spaces in the literature. Let $S(g, e;(\alpha_1,\beta_1), (\alpha_2,\beta_2), ... , (\alpha_r,\beta_r))$ denote the unique fully orientable Seifert fiber space with a genus $g$ base surface, Euler class $e$ (as an $S^1$ bundle), and $r$ exceptional fibers of type $\frac{\beta_i}{\alpha_i}$. 

Heegaard splittings of Seifert fiber spaces provide a rich set of examples which may be distinguished using Nielsen classes. In \cite{MS}, it is shown that irreducible Heegaard splittings of Seifert fiber spaces are either horizontal or vertical. Vertical heegaard splittings are well distinguished by the Neilsen classes they induce. On the other hand, it was shown in \cite{JJ2} that there are Seifert fiber spaces which admit infinitely non-isotopic horizontal Heegaard splittings, which nevertheless induce Nielsen equivalent generating sets. Due to the nature of this paper, we focus our attention on vertical Heegaard splittings.

To construct a vertical Heegaard splitting, we start by describing a graph in a Seifert fiber space  with $r$ exceptional fibers. The reader is encouraged to follow Figure \ref{fig:VerticalHeegaardSplittingGenus2} which contains an example of such a graph. Let $\Sigma$ be the base surface and $e_1,..., e_r$ be the images of the exceptional fibers, $f_1,..., f_r$ on $\Sigma$.  Choose a base point, $p$, on $\Sigma$ which is the image of a regular fiber.  Choose some non-empty, proper subset of indices  $\{ i_1,...,i_j \} \subset \{ 1,...,r \}$, and let $\sigma_{i_k}$ be an arc based at $p$ which joins $p$ to $e_{i_k}$. Let $\{ m_1,...,m_{r-j} \} = \{ 1,...,r \}\backslash \{ i_1,...,i_j \}$, be the complementary set to $\{ i_1,...,i_j \}$ and  let $q_{m_k}$ be a loop based at $p$ which winds around $e_{m_k}$ once. Finally, let $a_1, b_1, ... ,a_g,b_g$ be the usual collection of curves based at $p$ which cut $\Sigma$ into a disk. Choose all curves so that they are disjoint, except for at $p$. Let $\Gamma(i_1,...,i_j)$ be the graph consisting of $a_1,b_1,...,a_g,b_g, \sigma_{i_1},f_{i_1},...,\sigma_{i_j},f_{i_j}, q_{m_2},..., q_{m_{r-j}}$. Note that $q_{m_1}$ was excluded, this was arbitrarily chosen and any $q_{m_k}$ may be excluded. 

We define $\bm{H_1(i_1,...,i_j)}$ to be a tubular neighborhood of $\Gamma(i_1,...,i_j)$ in $S(g, e;(\alpha_1,\beta_1), ... , (\alpha_r,\beta_r))$. This is clearly a handlebody of genus $2g+m-1$. Let $\bm{H_2(i_1,...,i_j)}$ be the closure of the complement of $H_1(i_1,...,i_j)$. It is well known that $H_2(i_1,...,i_j)$ also a handlebody. We refer the reader to section 2 of \cite{LM} for a proof of this fact. Moreover, in Remarks 2.1, 2.2, and 2.3 of \cite{LM}, it is argued that the isotopy class of this Heegaard splitting only depends on the choices of $\{ i_1,...,i_j \}$ and that complementary choices of sets give the same Heegaard splitting. We therefore denote by $\bm{\Sigma(i_1,...,i_j)}$ the Heegaard splitting described above. Any Heegaard splitting obtained by the above process is called a \textbf{vertical Heegaard splitting}. A straightforward counting argument shows that $S(g, e;(\alpha_1,\beta_1), ... , (\alpha_r,\beta_r))$ has at most $2^{r-1}-1$ distinct vertical Heegaard splittings of genus $2g+m-1$. The following is a special case of Theorem 2.8 in \cite{LM} well suited for our purposes.

\begin{figure}
\label{fig:VerticalHeegaardSplittingGenus2}
\begin{center}
\includegraphics[scale=.35]{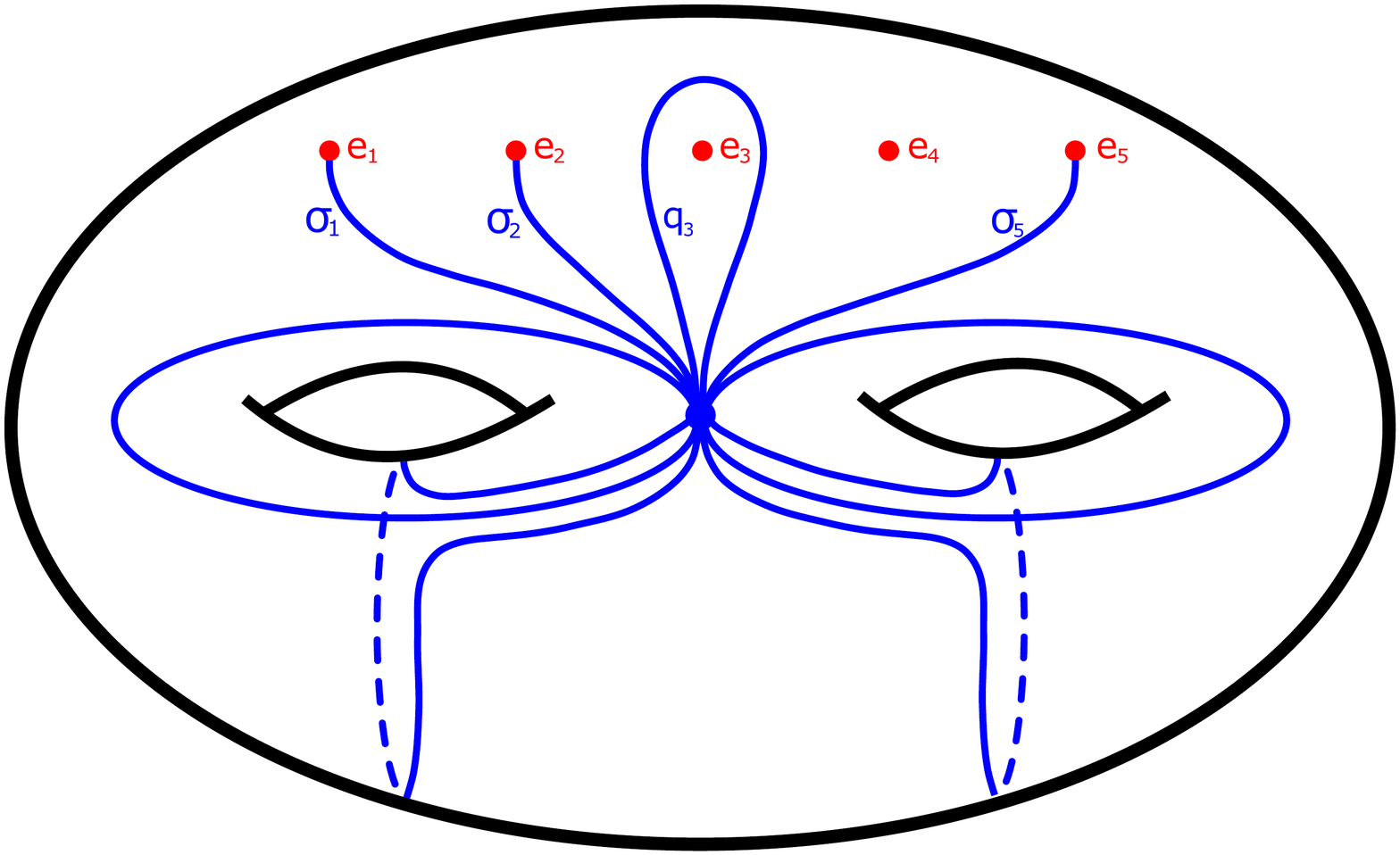}
\end{center}

\caption{The graph describing the vertical Heegaard Splitting $\Sigma(1,2,5)$. Note that the exceptional fibers $f_1$, $f_2$, and $f_5$ are included in this graph, but are not pictured.}
\end{figure}

\begin{theorem}

\label{thm:nonEqivNielsenClasses}

Let M be a Seifert fiber space with Seifert invariants $S(g, e;(\alpha_1,\beta_1), ... , (\alpha_r,\beta_r))$ satisfying the following conditions:
\begin{itemize}
\item $g>0$ and $r>0$, or $r \geq 3$ 
\item $\beta_i \not\equiv \pm 1 (mod\, \alpha_i)$ for all $i \in \{1,...,r \}$
\item All of the $\alpha_i$ are odd, pairwise relatively prime, and pairwise distinct.
\end{itemize}
Then for $k \in \{1,2 \}$, $\mathscr{N}(H_k(i_1,...,i_j)) = \mathscr{N}(H_k(k_1,...,k_j))$ if and only if the sets $\{ i_1, ... , i_j \}$ and $\{ k_1, ... , k_j \}$ are equal or complementary to each other in $\{1,2,...,m \}$
\end{theorem}

Under the conditions of the theorem above, we conclude that $S(g, e;(\alpha_1,\beta_1), ... , (\alpha_r,\beta_r))$ has exactly $2^{r-1}-1$ distinct vertical Heegaard splittings. Moreover, case (b) of Theorem 1.1 of \cite{BZ} states that the Heegaard genus of these manifolds is the same as the rank of their fundamental group, and both are equal to $2g+m-1$. Applying Meier's spin construction to the Heegaard splitting $\Sigma(i_1,...,i_j)$ therefore produces a $(3(2g+m-1), (2g+m-1))$-trisection of $S(\Sigma(i_1,...,i_j))$. By fundamental group considerations, we see that these are minimal genus trisections of  $S(\Sigma(i_1,...i_j))$. We distill the results stated above in the following corollary.

\nonIsoTri

\subsection{Non-diffeomorphic trisections}
We next turn our attention to diffeomorphism classes of trisections. We seek to distinguish two trisections up to diffeomorphism by the generating sets of the fundamental group they induce. Here, Proposition \ref{prop:triIsoImpliesNE} tells us that, in addition to the Nielsen transformations, we must also consider the effect of applying an arbitrary automorphism of the group to one of the sets of generators. We continue to consider the fundamental groups of a Seifert fiber space. It is well known that the fundamental group of a Seifert fiber space modulo its center is a Fuchsian group. Since Fuchsian groups are easier to work with in general, we seek to pass to this quotient.

One can check that images of Nielsen classes are well defined when passing to a quotient, and that if two systems are Nielsen equivalent, their images remain equivalent in a quotient. Therefore, given a group $H$, and two Nielsen classes of $H$, $M$ and $N$, we obtain two Nielsen classes of the group modulo its center, $[M]$ and $[N]$, and if $[M] \neq [N]$, $M \neq N$. Recall that any automorphism of $H$ descends to the quotient of $H$ by its center, $Z(H)$. Therefore, given an automorphism, $f$, of $H$, we additionally obtain well defined Nielsen classes of $H / Z(H)$, $f([M])$ and $f([N])$, and if $f([M]) \neq [N]$, then $f(M) \neq N$. The following theorem of Lustig, Moriah, and Rosenberger in \cite{LMR} is therefore pertinent.

\begin{theorem}
\label{thm:fuchsianLift}
Let $m \geq 2$ be an integer and let $\alpha_i$ be a set of $m$ pairwise distinct integers. Let G be the Fuchsian group given by the presentation $$\langle s_1, ... ,s_n, a_1, b_1, ... , a_g, b_g | s_i^{\alpha_i}, s_1s_2...s_n\Pi[a_i,b_i] \rangle$$

Consider the generating system given by $\{ x_1 = s_1^{\beta_1},..., x_{k-1}^{\beta_{k-1}}, x_{k+1}^{\beta_{k+1}},...,x_n = s_n^{\beta_n}, a_1,b_1,...a_g,b_g \}$ where $(\alpha_i, \beta_i)=1$. Then any automorphism $h: G \to G$ is induced by some automorphism of the free group $F_{2g+m-1} = F[X_1,...,X_{k-1}, X_{k-1} ,..., X_n, A_1, B_1, ... ,A_g, B_g]$ with respect to the surjection $F_{2g+m-1} \to G$ given by $X_i \mapsto x_i$, $A_i \mapsto a_i$, $B_i \mapsto b_i$.
\end{theorem}

By a theorem of Rosenberger (\cite{GR}, Satz 2.2), all generating systems of a Fuchsian group are Nielsen equivalent to one of the form described in Theorem \ref{thm:fuchsianLift} above. It follows that two systems in a Fuchsian group satisfying the hypothesis of Theorem \ref{thm:fuchsianLift} are Nielsen equivalent after an automorphism if and only if they are Nielsen equivalent. By Proposition \ref{prop:triIsoImpliesNE}, we see that all of the examples of non-isotopic trisections constructed as a corollary of Theorem \ref{thm:nonEqivNielsenClasses} are in fact non-diffeomorphic. We summarize these results in the following corollary.

\nonDiffTri

\subsection{Spun small Seifert fiber spaces}

The main goal of this section is to exhibit non-isotopic trisections which become isotopic after a single stabilization. Here, a simple class of Seifert fiber spaces proves amenable to study. Seifert fiber spaces with 3 exceptional fibers whose base space is a sphere are called \textbf{small Seifert fiber spaces}. These manifolds admit at most 3 vertical Heegaard splittings of genus 2. In \cite{BCZ}, the authors show that if for every singular fiber $\beta_i \not\equiv \pm 1 (mod\, \alpha_i)$, then $S(0,e;(\alpha_1,\beta_1), (\alpha_2,\beta_2),(\alpha_3,\beta_3))$ admits exactly 3 Heegaard splittings up to isotopy, all of which are vertical and distinguished by their Nielsen classes. 

These manifolds also admit a genus 3 Heegaard splitting obtained as follows: Take a path on the base space from a base point to each of the 3 exceptional points and connect each path to the respective  exceptional fiber to form a wedge of 3 circles. Let $H_1(3)$ to be a regular neighborhood of this graph and let $H_2(3)$ be the complement of the interior of $H_1$. Note that the graph we constructed naturally forms the spine for $H_1(3)$. We call this Heegaard splitting $\Sigma_3$.

This construction mirrors the construction of a vertical Heegaard splitting and one readily sees that the spines of $H_1(i,j)$ are subgraphs of the spine of $H_1(3)$. Two of these Heegaard splittings together with the genus 3 splitting are depicted in Figure \ref{fig:SmallSeifertSplittings}. The following corollary of the classification of Heegaard splittings of handlebodies in \cite{ST} tells us that these splittings are closely related.

\begin{figure}
\label{fig:SmallSeifertSplittings}
\begin{center}
\includegraphics[scale=.5]{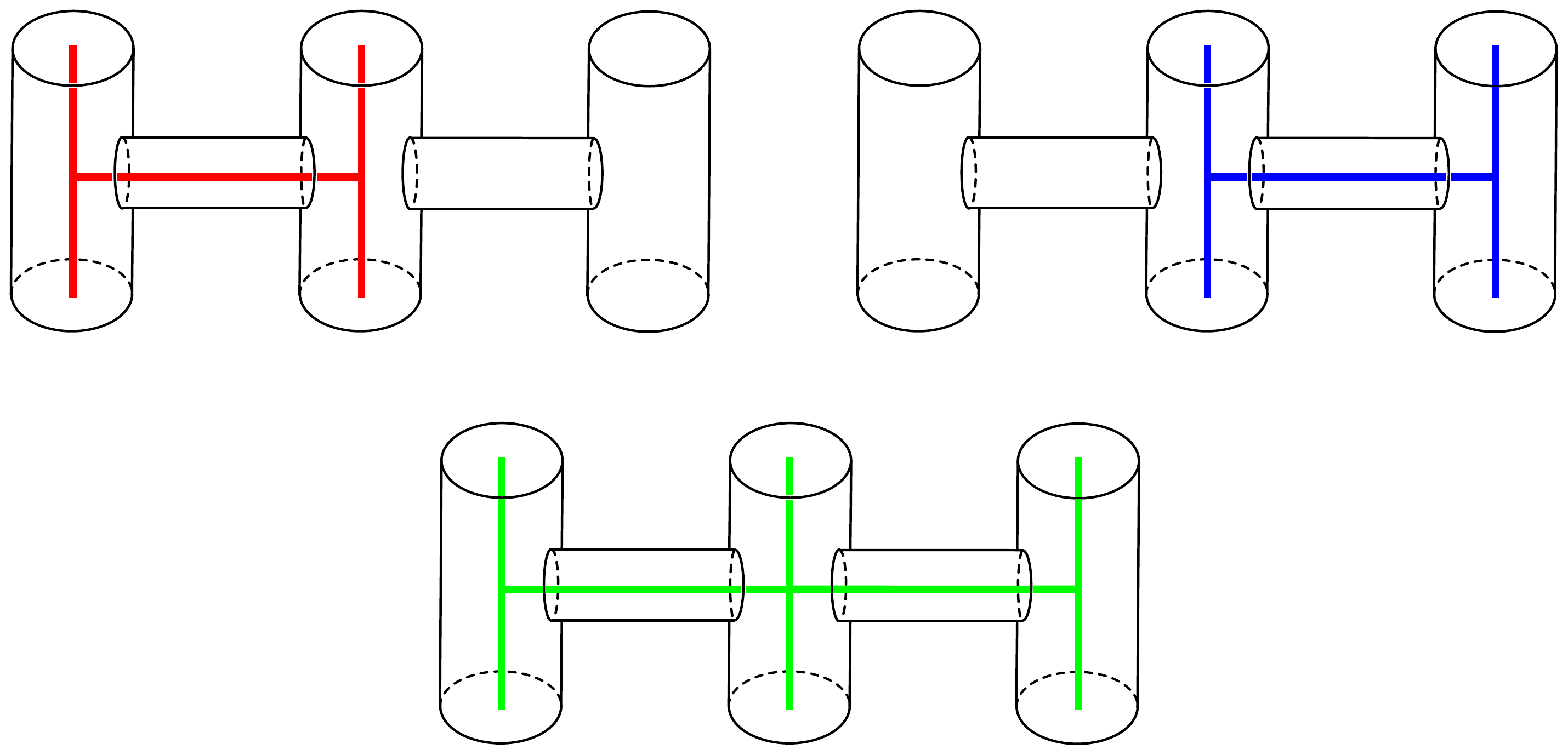}
\end{center}

\caption{Top: A schematic picture of spines of $H_1$ of the vertical Heegaard splittings of S(0,3). Horizontal tubes lie in a neighborhood of the base sphere while vertical tubes are neighborhoods of fibers. Bottom: A spine of a genus 3 handlebody which contains the spines of both of the vertical Heegaard splittings above.}
\end{figure}

\begin{corollary}
Let $H_1 \cup_\Sigma H_2$ and $H_1' \cup_{\Sigma'} H_2'$ be two Heegaard splittings of $M^3$ and let $S$ and $S'$ be spines of $H_1$ and $H_1'$ respectively. Then if $S$ is a subgraph of $S'$ then  $H_1' \cup_{\Sigma'} H_2'$ is a stabilization of $H_1 \cup_\Sigma H_2$.
\end{corollary}

We therefore conclude that $\Sigma_3$ is a stabilization of $\Sigma(i,j)$ for all $\{ i,j \} \subset \{1,2,3 \}$. In otherwords, all 3 of the vertical Heegaard splittings of  $S(0,e; \allowbreak (\alpha_1,\beta_1), \allowbreak (\alpha_2,\beta_2), \allowbreak(\alpha_3,\beta_3))$ become isotopic after a single stabilization. An easy analysis of the local diagrammatic modifications in \cite{JM} used to pass between a Heegaard splitting and its spin shows that the spin of a stabilized Heegaard splitting is a stabilized trisection. This implies that $S(\Sigma(1,2))$, $S(\Sigma(2,3))$, and $S(\Sigma(1,3))$ are all non-isotopic trisections which become pairwise isotopic after a single (balanced) stabilization. The following proposition shows that a balanced stabilization as opposed to an unbalanced stabilization is required in a very strong sense. We refer the reader to \cite{MSZ} for the details of unbalanced trisections and unbalanced stabilizations.

\begin{proposition}
The trisections $S(\Sigma(1,2))$, $S(\Sigma(2,3))$, and $S(\Sigma(1,3))$ become pairwise isotopic after a single balanced stabilization, however, for any two $\{ k,l \} \subset \{ 1,2,3 \}$ these trisections remain pairwise non-isotopic after any sequence of $k-$ and $l-$stabilizations.
\end{proposition}

\begin{proof}
Let $H_1^{i,j}$ be $H_1$ of the Heegaard splitting $\Sigma(i,j)$ and let $X_n^{i,j}$ be $X_n$ of $S(\Sigma(1,2))$.  By Lemma \ref{lem:spinIsANE} we have that for all $n \in \{ 1,2,3 \}$, $\mathscr{N}(X_n^{i,j}) = H_1^{i,j}$. Choose any $\{ k,l \} \subset \{ 1,2,3 \}$ and let $m$ be the remaining index. Then under any sequence of $k-$ and $l-$stabilizations, $X_m^{i,j}$ is unchanged. In particular, if we take $\{ i,j \} \neq \{ i',j' \}$ then under any sequence of $k-$ and $l-$ stabilizations $\mathscr{N}(X_m^{i,j}) \neq \mathscr{N}(X_m^{i',j'})$. By Proposition \ref{prop:triIsoImpliesNE} this implies that the trisections remain non-isotopic.
\end{proof}

\section*{Acknowledgments}

The author would like to thank Slava Krushkal and Jeff Meier for helpful conversations which clarified many of the arguments in this paper. The author was partially supported by NSF grant DMS-1612159.

\bibliography{mybib}
	\bibliographystyle{plain}
\end{document}